\documentclass[11pt,reqno]{amsart} \usepackage{amssymb, mathrsfs, enumerate}
\usepackage{microtype}
\usepackage{fontenc}
\usepackage{xurl, colonequals, bbold}
\usepackage[margin=0.8in,headheight=13.6pt]{geometry}

\usepackage{mathtools}
\usepackage{quiver}
\usepackage[mathlines]{lineno}

\usepackage[
    style=alphabetic,
    backend=biber,
    giveninits=true,
    doi=false,
    isbn=false,
    url=false,
    eprint=false
]{biblatex}

\DeclareNameAlias{default}{given-family}
\DeclareNameAlias{sortname}{given-family}

\DeclareFieldFormat[article,periodical]{title}{\mkbibitalic{#1}}
\DeclareFieldFormat[book,inbook,incollection]{title}{\mkbibitalic{#1}}
\DeclareFieldFormat{journaltitle}{#1}

\renewbibmacro{in:}{%
  \ifentrytype{article}{}{\printtext{\bibstring{in}\intitlepunct}}}

\AtEveryBibitem{\clearfield{pages}}
\AtEveryCitekey{\clearfield{pages}}
\renewbibmacro*{volume+number+eid}{%
  \printfield{volume}%
  \setunit*{\addspace}%
  \printfield{number}%
  \setunit{\addcomma\space}%
}

\usepackage[all]{xy}
\usepackage{tikz-cd}
\usepackage[colorlinks=true,bookmarksnumbered=true,pdfpagemode=None,extension=pdf]{hyperref}
\hypersetup{
	colorlinks,
	linkcolor={red!50!black},
	citecolor={green!50!black},
	urlcolor={blue!80!black}
}
\hypersetup{pdfauthor={Ayan Nath}}

\newtheorem{theorem}{Theorem}[subsection]  
\renewcommand\thetheorem{%
  \ifnum\value{subsection}>0
    \thesubsection.\arabic{theorem}%
  \else
    \thesection.\arabic{theorem}%
  \fi
}
\newtheorem{proposition}[theorem]{Proposition}  
\newtheorem{lemma}[theorem]{Lemma}  
\newtheorem{corollary}[theorem]{Corollary}  
\newtheorem{conjecture}[theorem]{Conjecture}  
\newtheoremstyle{nonitalic}
  {10pt} 
  {10pt} 
  {\normalfont} 
  {} 
  {\bfseries} 
  {.} 
  {5pt plus 1pt minus 1pt} 
  {} 

\theoremstyle{nonitalic}
\newtheorem{example}[theorem]{Example}  
\newtheorem{remark}[theorem]{Remark}  
\newtheorem{assumption}[theorem]{Assumption}  

\DeclareMathOperator{\Aut}{Aut}

\DeclareMathOperator{\Out}{Out}

\DeclareMathOperator{\Dyn}{Dyn}

\DeclareMathOperator{\SL}{SL}

\DeclareMathOperator{\Spec}{Spec}

\DeclareMathOperator{\Mat}{Mat}

\DeclareMathOperator{\ad}{ad}

\renewcommand{\ge}{\geqslant}

\renewcommand{\le}{\leqslant}

\newcommand{\Mod}[1]{\ (\mathrm{mod}\ #1)}
\newcommand{\ii}{\item}
\newcommand{\mf}{\mathfrak}
\newcommand{\inj}{\hookrightarrow}

\newcommand{\ol}{\overline}
\newcommand{\ul}{\underline}

\newcommand{\del}{\partial}

\DeclareSymbolFont{arrows}{OMS}{cmsy}{m}{n}
\DeclareMathSymbol{\leftrightarrow}{\mathrel}{arrows}{"24}
\DeclareMathSymbol{\leftarrow}{\mathrel}{arrows}{"20}
   
\DeclareMathSymbol{\rightarrow}{\mathrel}{arrows}{"21}
   \let\to=\rightarrow
\DeclareMathSymbol{\mapstochar}{\mathrel}{arrows}{"37}
\DeclareMathSymbol{\relbardash}{\mathbin}{arrows}{"00}

\newcommand{\OO}{\mathcal O}

\newcommand{\ZZ}{\mathbb Z}
\renewcommand{\AA}{\mathbb A}
\newcommand{\PP}{\mathbb P}

\newcommand{\RR}{\mathbb R}
\newcommand{\GG}{\mathbb G}
\newcommand{\QQ}{\mathbb Q}

\newcommand{\HH}{\mathrm H} 

\newcommand{\tensor}{\otimes}

\usepackage{stmaryrd}
\usepackage[perpage]{footmisc}
\usepackage[T1]{fontenc}
\renewcommand{\mathbb}[1]{\mathbf{#1}}
\DeclareMathSymbol{G}{\mathalpha}{operators}{`G}
\DeclareMathSymbol{B}{\mathalpha}{operators}{`B}
\DeclareMathSymbol{T}{\mathalpha}{operators}{`T}
\DeclareMathSymbol{P}{\mathalpha}{operators}{`P}
\DeclareMathSymbol{V}{\mathalpha}{operators}{`V}
\DeclareMathSymbol{Z}{\mathalpha}{operators}{`Z}
\newcommand{\G}{\ensuremath{\mathbf G}}
\newcommand{\B}{\ensuremath{\mathbf B}}
\newcommand{\T}{\ensuremath{\mathbf T}}
\newcommand{\U}{\ensuremath{\mathbf U}}
\newcommand{\XX}{\ensuremath{\mathbb X}}
\newcommand{\mU}{\ensuremath{\mathrm U}}
\newcommand{\torus}{\ensuremath{\GG_{\mathrm m,k}^2}}

\newcommand{\Gqs}{\ensuremath{\mathbf G_{\text{q-\'ep}}}}

\renewcommand{\ad}{\mathrm{ad}}
\renewcommand{\P}{\ensuremath{\mathbf P}}
\newcommand{\M}{\ensuremath{\mathbf M}}
\newcommand{\fil}{\operatorname{fil}}

\begin{document}

\title{Compactification of reductive group schemes}

\begin{abstract}
    Let $G$ be 
    an isotrivial reductive group over a scheme $S$. 
    We construct a smooth projective
    $S$-scheme containing $G$ as a fiberwise-dense open subscheme
    equipped with left and right actions of $G$ which extend the translation actions of $G$ on itself.
    This verifies a conjecture of Česnavičius \cite{kestutis}.
    When $G$ is adjoint, we recover fiberwise the wonderful compactification.
    Finally, we give an example of a non-isotrivial torus admitting no equivariant compactification.
\end{abstract}

\author{Ayan Nath}
\address{Massachusetts Institute of Technology, Cambridge, MA, USA}
\email{ayannath@mit.edu}

\date{January 17th, 2026}

\maketitle

\section{Introduction}

A reductive group scheme $G$ over a base scheme $S$ is called {\em locally
isotrivial} if each point $s \in S$ has a Zariski open neighborhood $U_s \inj
S$ admitting a finite \'etale cover $U'_s \to U_s$ such that $G\times_S U'_s$
is split.
This is equivalent to local isotriviality of the central torus \cite[Expos\'e XXIV, Théorème
4.1.5]{sga3}.
A reductive group is called {\em isotrivial} if $U_s$ can be chosen to be $S.$

Česnavičius  conjectured that every isotrivial reductive group scheme $G$ has a compactification equipped with a left action of $G$ extending that of $G$ on itself:

\begin{conjecture}[\protect{\cite[Conjecture 6.2.3]{kestutis}}]\label{conj:main}
For an isotrivial reductive group $G$ over a Noetherian scheme $S$, there are a projective, finitely presented $S$-scheme $\overline{G}$ equipped with a left $G$-action and a $G$-equivariant $S$-fiberwise dense open immersion
\[
    G \hookrightarrow \overline{G}.
\]
\end{conjecture}

The case where $G$ is an isotrivial torus was established in \cite[\S 6.3]{kestutis} (see also \cite[Proposition 7.2.1]{kundu}). For adjoint $G$, Li \cite{shangli} confirmed the conjecture using a variant of the Artin--Weil method of birational group laws.
In this paper, we construct a smooth
projective $G\times_S G$-equivariant compactification for any
isotrivial reductive group $G$ over an arbitrary base $S$,
thus verifying
Conjecture \ref{conj:main}:

\begin{theorem}\label{thm:main}
    Let $G$ be an isotrivial reductive group scheme over a scheme $S.$
    Then there exists a finitely presented smooth projective $S$-scheme
    $\ol G$ containing $G$ as a
    fiberwise-dense open subscheme equipped with a left and right action of $G$
    extending that on $G$ given by left and right multiplication.
\end{theorem}

When $S$ is the spectrum of an algebraically closed field and $G$ is adjoint, it is known that the wonderful compactification satisfies the assertions of Theorem \ref{thm:main}.
To place our results in context:
Over algebraically closed fields, the wonderful compactification for adjoint reductive groups was introduced by de Concini and Procesi \cite{DP}, and extended to arbitrary characteristic by Strickland \cite{strickland}. According to \cite{shangli}, Gabber constructed it for Chevalley group schemes over $\Spec\ZZ$ in an unpublished article.

Most recently, Li \cite{shangli} constructed the wonderful compactification for adjoint reductive group schemes over arbitrary bases. 
For adjoint $G$ and $\ol G$ as in Theorem \ref{thm:main}, Li's method relies on the existence of a `big cell' in $\overline{G}$. Assuming $G$ is split, this cell is an open dense subscheme $\overline{\Omega} \subset \overline{G}$ isomorphic to the affine space $\mU^-\times_S \AA_S^{\ell}\times_S \mU^+$, where $\ell$ is the rank of $G$ and $\mU^\pm$ are the unipotent radicals of opposite Borel subgroups. Since the $G\times_S G$-translates of $\overline{\Omega}$ cover $\overline{G}$, one can expect to recover $\overline{G}$ as a suitable quotient of $G\times_S \overline{\Omega} \times_S G$.

However, it is unclear how to generalize this approach to the non-adjoint case.
Our approach is inspired by the work of Martens--Thaddeus \cite{thaddeus}, who used Vinberg monoids to construct compactifications over algebraically closed fields of characteristic $0$.
Since we are unaware of a reference, we take this opportunity to record the theory of Vinberg monoids for reductive group schemes over arbitrary bases in \S\ref{sec:vinberg}.

When $G$ is semisimple, we can do more:

\begin{theorem}\label{thm:semisimple}
    Let $G$ be a semisimple reductive group scheme over $S$. Then there exist finitely presented projective $S$-schemes $\overline{G}$ and $\overline{G}_{\mathrm{ad}}$ containing $G$ and $G_{\mathrm{ad}}$, respectively, as fiberwise-dense open subschemes, equipped with left and right actions of $G$ and $G_{\mathrm{ad}}$ extending those on $G$ and $G_{\mathrm{ad}}$ by themselves such that
    \begin{itemize}
        \item 
    there is an equivariant morphism $\overline{G} \to \overline{G}_{\mathrm{ad}}$ extending the central isogeny $G \to G_{\mathrm{ad}}$,
    \ii for each geometric point $s \in S$, $(\ol G)_s$ is the normalization of $(\ol G_\ad)_s$ in $G_s,$
        \item $\overline{G}_\ad$ is $S$-smooth and agrees with the de Concini--Procesi wonderful compactification over geometric points of $S$, and
        \item its boundary $\overline{G}_\ad \setminus G_\ad$ is the union of $S$-smooth relative effective Cartier divisors with relative normal crossings\footnote{Following \cite[Expos\'e XIII, \S2.1]{SGA1}, for an $S$-scheme $X$, we say that a relative Cartier divisor $D \subset X$ is \emph{strictly with relative normal crossings} if there exists a finite family $(f_i \in \Gamma(X, \mathcal{O}_X))_{i \in I}$ such that (1) $D = \bigcup_{i \in I} V_X(f_i)$, and (2) for every $x \in \operatorname{Supp}(D)$, $X$ is smooth at $x$ over $S$, and the closed subscheme $V((f_i)_{i \in I(x)}) \subset X$ is smooth over $S$ of codimension $|I(x)|$, where $I(x) = \{i \in I \mid f_i(x) = 0\}$. The divisor $D$ has \emph{relative normal crossings} if \'etale locally on $X$ it is strictly with relative normal crossings.}.
    \end{itemize}
\end{theorem}

All our constructions commute with base change on $S$. We also verify that our construction agrees with that of \cite{shangli} for adjoint $G$ (Proposition \ref{prop:agree-with-shangli}).

Finally, we show that the isotriviality assumption in Conjecture \ref{conj:main} is essential. By analyzing the standard example of a non-isotrivial torus over a nodal rational curve given in \cite[Expos\'e X, \S1.6]{sga3}, we obtain the following negative result:

\begin{theorem}\label{thm:counterexample}
    Let $S$ be the nodal cubic curve over an algebraically closed field, and let $T$ be a non-isotrivial torus over $S$ (constructed in \S\ref{sec:bad-torus}). Then there exists no projective $S$-scheme $\overline{T}$ containing $T$ as a fiberwise-dense open subscheme such that the left action of $T$ on itself extends to $\overline{T}$.
\end{theorem}

The main difficulty in the proof of Theorem \ref{thm:counterexample} is that a potential compactification $\overline{T} \to S$ need not have normal fibers, preventing the direct application of the classification theory of toric varieties. We overcome this by using Brion's results on the linearization of line bundles \cite{brion} to reduce the problem to the setting of toric schemes over discrete valuation rings, where we derive a contradiction.

\begin{remark}
    Although we work with base schemes, it is easily seen that proofs of Theorems \ref{thm:main} and \ref{thm:semisimple} go through even when $S$ is an algebraic space.
\end{remark}

\subsection*{Acknowledgments} 
The author is grateful to Bjorn Poonen and K\k{e}stutis
\v{C}esnavi\v{c}ius for helpful conversations.
The author also thanks Shang Li and Arnab Kundu for their comments on earlier drafts.
This work was supported in part by Simons Foundation grant \#402472 to Bjorn Poonen.

\section{Affine monoids}\label{sec:vinberg}
\subsection{Vinberg monoids} 
Let $\G$ be a split reductive group scheme over a connected base scheme $S$.
Let $\T$ be the abstract Cartan-- 
this may be defined as $\B/\mathrm R_{\mathrm u}(\B)$ for any $S$-Borel $\B\subset \G$. The key point is that this does not depend on the choice of $\B$.
Define $\G_+ \colonequals \T \times^{Z_\G}_S \G$ and $\T_\ad \colonequals \T/Z_\G$, where $Z_\G$ denotes the center of $\G$.
Let $\Delta\subset \XX^\bullet(\T_\ad)$ be the set of simple roots.
Note that $\Delta$ is canonically defined as a {\em subset} of $\XX^\bullet(\T_\ad)$.
Let $\T_\ad \inj \T_\ad^+ \colonequals \Spec \OO_S[\XX^\bullet(\T_\ad)_{\text{pos}}]$ be the obvious toric embedding where $\XX^\bullet(\T_\ad)_{\text{pos}}$ is the submonoid generated by $\Delta.$
Here, $\T_\ad^+$ is viewed as an $S$-monoid with unit group $\T_\ad.$
The {\em Vinberg monoid} is a certain reductive monoid scheme $V_\G$ over $S$ equipped with an $S$-monoid homomorphism $\mf a \colon V_\G \to
\T_{\ad}^+$ called the {\em abelianization}.
It is true that $\T^+_{\ad}$ is the largest commutative monoid quotient of $V_\G$ in a precise sense.
The properties relevant to us are recorded in Theorem \ref{thm:vinberg-main}.

Let us briefly recall the construction of the Vinberg monoid given in \cite[\S3.2]{zhu} over an algebraically closed field.
The same
    works over $\ZZ$ and hence for any split reductive group over
    an arbitrary base. Assume $S = \Spec \ZZ$ without any loss of generality.
Let $\XX^\bullet(\T)^+$ be the submonoid of dominant characters, and let $\XX^\bullet(\T)^+_{\text{pos}}$ be the submonoid generated by $\XX^\bullet(\T)^+$ and $\XX^\bullet(\T_\ad)_{\text{pos}}$.
We equip $\mathbb{X}^\bullet(\mathbf{T})$ with the partial order $\preceq$ defined by $\lambda \preceq \mu$ if $\mu - \lambda \in \mathbb{X}^\bullet(\mathbf{T}_{\mathrm{ad}})_{\mathrm{pos}}$.
We write $\lambda \prec \mu$ if $\lambda \preceq \mu$ and $\lambda \ne \mu$.

The coordinate ring $\OO(\G)$ admits a canonical multi-filtration indexed by $\XX^\bullet(\T)_{\text{pos}}^+$, induced by the $\G\times \G$ action on $\OO(\G)$ via left and right translation, given by setting $\fil_\nu \OO(\G),\, \nu \in \XX^\bullet(\T)_{\text{pos}}^+$, as the maximal $\G\times \G$-submodule of $\OO(\G)$ such that all its weights $(\lambda, \lambda') \in \XX^\bullet(\T) \times \XX^\bullet(\T)$ satisfy $\lambda \preceq -w_0(\nu)$ and $\lambda' \preceq \nu$. 
Here, $w_0$ is the longest element of the Weyl group of $\G.$
Each piece $\fil_\nu \OO(\G)$ is finite free over $\ZZ$ and the associated graded is given by 
\[\operatorname{gr} \OO(\G) \colonequals \bigoplus_{\nu \in \XX^\bullet(\T)^+_{\text{pos}}} \frac{\fil_{\nu} \OO(\G)}{\sum_{\lambda \prec \nu} \fil_\lambda \OO(\G)} =  \bigoplus_{\nu \in \XX^\bullet(\T)^+} \mathrm S_{-w_0(\nu)}\tensor \mathrm S_{\nu},\]
where $\mathrm S_\nu$ denotes the Schur module of highest weight $\nu$, i.e., the induced $\G$-module from the character $-\nu$ of $\B$ \cite[Lemma 3.2.1 (4)]{zhu}.

The Vinberg monoid $V_{\G}$ is defined as the spectrum of the Rees algebra associated to this filtration:
\[
    V_{\G} \colonequals \Spec\left( \bigoplus_{\nu \in \XX^\bullet(\T)_{\text{pos}}^+} \fil_\nu \OO(\G)\right),
\]
endowed with the natural (co)multiplication map.
It is an affine monoid $\ZZ$-scheme of finite type which admits 
a monoid homomorphism $\mf a: V_{\mathbf{G}} \to \mathbf{T}_{\mathrm{ad}}^+ \colonequals \Spec \mathbb{Z}[\mathbb{X}^\bullet(\mathbf{T}_{\mathrm{ad}})_{\mathrm{pos}}]$ induced by the inclusion of the ring $\mathbb{Z}[\mathbb{X}^\bullet(\mathbf{T}_{\mathrm{ad}})_{\mathrm{pos}}]$ into the Rees algebra.
The formation of $(V_\G, \mf a)$ commutes with base change on $S$. Indeed, for any ring $R$, the injection of filtered $R$-algebras $R\tensor \sum_{\nu}\fil_\nu \OO(\G) \inj \sum_{\nu}\fil_\nu \OO(\G_R)$ induces an isomorphism on associated graded algebras since the formation of Schur modules commutes with base change. Thus, $R\tensor \sum_{\nu}\fil_\nu \OO(\G) \inj \sum_{\nu}\fil_\nu \OO(\G_R)$ is actually an isomorphism of filtered rings.

\begin{theorem}\label{thm:vinberg-main}
    The $S$-monoid $V_\G$ fits into a Cartesian diagram
\[\begin{tikzcd}
    {\G_+} & {V_{\G}} \\
    {\T_{\ad}} & \T_\ad^+
	\arrow[hook, from=1-1, to=1-2]
	\arrow[from=1-1, to=2-1]
	\arrow["{\mf a}", from=1-2, to=2-2]
	\arrow[hook, from=2-1, to=2-2]
\end{tikzcd}\]
and
    \begin{itemize}
        \ii $\G_+$ is the unit group of $V_\G$,
        \ii the $\T\times_S \G \times_S \G$ action on $\G_+$ given by $\T$
        acting on the first component and $\G$ acting on itself by left and
        right multiplication extends to an action
        on $V_\G$,
        \ii $\mf a$ is faithfully flat and equivariant.
    \end{itemize}
\end{theorem}

\begin{proof}
        See \cite[Proposition 3.2.2]{zhu}.
\end{proof}

\subsection{Nondegenerate
locus}\label{sec:nondegenerate-locus}
Choose an $S$-Borel $\B$ and a Cartan {\em sub}group $c\colon \T \inj \B$.
Let $\U_+$ be the unipotent radical of $\B$ and $\U_-$ be that of the
opposite Borel.

\begin{proposition}
There is a {\em canonical section} $\ol{\mf s}\colon \T_\ad^+ \to V_\G$
extending $\mf s \colon t
\mapsto (t,c(t)) \Mod{Z_\G}$ on respective unit groups.
\end{proposition}

\begin{proof}
We identify $\T$ with the image of $c$ and assume $S = \Spec \ZZ$. The restriction ring map $r: \OO(\G) \to \OO(\T)$ sends the filtered piece $\fil_\nu \OO(\G)$ into the span of characters $e^\lambda$ satisfying $\lambda \preceq \nu$. Here, $e^\lambda$ denotes the character corresponding to $\lambda$.

We define a retraction $\phi: \OO(V_{\G}) \to \OO(\T_{\ad}^+)$ by mapping a homogeneous element $f_\nu \in \fil_\nu \OO(\G)$ to the coefficient of its highest weight term. Explicitly, we define:
$$
\phi(f_\nu) = (\text{coefficient of } e^\nu \text{ in } r(f_\nu)) \cdot e^\nu.
$$
This map is multiplicative.
Consider homogeneous elements $f \in \fil_\nu \OO(\G)$ and $g \in \fil_\mu \OO(\G)$. Their product $fg$ lies in $\fil_{\nu+\mu} \OO(\G)$. The restriction map is a ring homomorphism, so $r(fg) = r(f)r(g)$. Since $r(f)$ involves only weights $\lambda \preceq \nu$ and $r(g)$ involves only weights $\lambda' \preceq \mu$, the weight $\nu+\mu$ in the product $r(f)r(g)$ can only arise from the product of the term $e^\nu$ in $r(f)$ and $e^\mu$ in $r(g)$. Indeed, if $\lambda + \lambda' = \nu + \mu$ with $\lambda \preceq \nu$ and $\lambda' \preceq \mu$, then $(\nu - \lambda) + (\mu - \lambda') = 0$. Since $\nu - \lambda$ and $\mu - \lambda'$ are non-negative integer linear combinations of simple roots, this implies $\lambda = \nu$ and $\lambda' = \mu$.

Thus, $\phi(fg) = \phi(f)\phi(g)$. Since $\phi$ maps the element $1 \cdot e^\nu \in \OO(V_{\G})$ to $e^\nu \in \OO(\T_{\ad}^+)$, it forms the required section $\ol{\mf s}$ of the abelianization.
\end{proof}

\begin{proposition}\label{prop:big-cell}
    The multiplication map $m\colon \U_-\times_S \T \times_S \T_\ad^+ \times_S
    \U_+ \to V_\G$
    given by $(u_-, t, a, u_+) \mapsto u_-t\ol{\mf s}(a)u_+$ is an open embedding.
\end{proposition}

\begin{proof}
    We first prove this in the case when $S$ is the spectrum of an algebraically closed field $k$. 
    By \cite[\S7.5]{stable}, $V_\G$ is a normal irreducible $k$-variety.
    Therefore, $m$ is a birational morphism of normal irreducible $k$-varieties.
    By Zariski's main theorem, it suffices to verify that $m$ is injective on $k$-points.
    Checking this is equivalent to showing that 
    \[u_-t\ol{\mf s}(a_1)u_+ = \ol{\mf s}(a_2)\implies u_-=u_+=t=1, a_1=a_2\]
    for $u_- \in \U^-(k),u_+\in \U^+(k), t\in \T(k), a_1,a_2\in \T^+_\ad(k)$.
    We may assume that $a_1 = e_I$ where $e_I$ is an idempotent corresponding to a subset $I$ of the set of simple roots without any loss of generality.
    There is a map from the source of $m$ to $\T^+_\ad$ given by $(u_-, t, a, u_+) \mapsto ta$. Then $m$
    is a $\T_\ad^+$-morphism where the target is viewed as a $\T_\ad^+$-scheme via $\mf a$.
    Therefore, $t e_I = a_2$ as elements of $\T^+_\ad.$
    Hence, \[u_- t \ol{\mf s}(e_I)u_+ = (t,t)\ol{\mf s}(e_I)\implies u_- \ol{\mf s}(e_I) u_+ = (1,t)\ol{\mf s}(e_I)\implies (t^{-1}u_-)\ol{\mf s}(e_I) u_+ = \ol{\mf s}(e_I).\]
    In \cite[Appendix C]{constant-term}, it is shown that the $\P\times_k \P^-$-stabilizer of $\ol{\mf s}(e_I)$, where $(\P,\P^-)$ is a certain pair of opposite parabolics of $\G$ depending on $I$, is $\P \times_\M \P^-$, where $\M$ is the Levi factor.
    It is also shown that 
\[
\P = \{ g \in \G \mid g \cdot \overline{\mathfrak{s}}(e_I) = \overline{\mathfrak{s}}(e_I) \cdot g \cdot \overline{\mathfrak{s}}(e_I) \} \quad \text{and} \quad \P^- = \{ g \in \G \mid \overline{\mathfrak{s}}(e_I) \cdot g = \overline{\mathfrak{s}}(e_I) \cdot g \cdot \overline{\mathfrak{s}}(e_I) \}.
\]
    From the above descriptions, one deduces that the $\G\times_k \G$-stabilizer of $\ol{\mf s}(e_I)$ is actually contained in the $\P\times_k\P^-$-stabilizer and hence equals $\P\times_\M \P^-$.
    We conclude that $t = u_- = u_+ = 1$ and hence $m$ is injective on $k$-points.
    
    Let us now consider the case of general $S$. Since the formation of $m$ commutes with base change, it follows by fibral criteria that $m$ is \'etale \cite[\href{https://stacks.math.columbia.edu/tag/039E}{Tag 039E}]{stacks}.
    Thus, $m$ is an \'etale monomorphism and we conclude by
\cite[\href{https://stacks.math.columbia.edu/tag/025G}{Tag
025G}]{stacks}.
\end{proof}

\begin{lemma}\label{lem:flatness-of-saturation}
    Let $H$ be a group scheme locally of
    finite presentation and flat over a scheme $S$ acting on
    an $S$-scheme $X.$ 
    Let $U \subseteq X$ be an
    open
    subscheme.
    Then the saturation $H\cdot U$ of $U$ is an $H$-stable open
    subscheme of $X.$ If in addition
    \begin{itemize} 
        \ii $U$ is
    $S$-flat then so is $H\cdot U$,
        \ii $U$ is locally of finite presentation over $S$
        then so is $H\cdot U$,
        \ii $U$ is $S$-smooth then so is $H\cdot
        U$.
\end{itemize}
\end{lemma}

\begin{proof}
    The $H$-saturation $H\cdot U$ is defined as the
    image of the
    composition
    \[H\times_S U \inj H\times_S X
        \xrightarrow{(h,x)\mapsto (h,h\cdot x)}
    H\times_S X \xrightarrow{(h,x)\mapsto x} X.\]
    The middle map is an isomorphism and the last
    map is a projection. Since flat morphisms of
    locally finite presentation are universally
    open \cite[\href{https://stacks.math.columbia.edu/tag/01UA}{Tag 01UA}]{stacks}, it follows that $H\cdot U$ is open.
    The induced map $H\times_S U \to H\cdot U$ is
    fppf. 
    Since flatness and being locally of finite
    presentation are fppf local on source \cite[\href{https://stacks.math.columbia.edu/tag/06ET}{Tag 06ET}, \href{https://stacks.math.columbia.edu/tag/06EV}{Tag 06EV}]{stacks}, $S$-flatness of
    $H\cdot U$ is equivalent to that of $H\times_S
    U$ and the same is true for being locally of finite
    presentation.
    When $U$ is smooth, we reduce to the case of
    $S$ an algebraically closed point by
    \cite[\href{https://stacks.math.columbia.edu/tag/01V8}{Tag
    01V8}]{stacks}, in which case smoothness follows
    because translations of $U$ cover $H\cdot U.$
\end{proof}

The $\G\times_S\G$-saturation of the image of $m$ in Proposition \ref{prop:big-cell} is called the {\em nondegenerate locus}, denoted $V_\G^\circ.$
It is independent of the choice of Borel and Cartan
subgroup.

\begin{corollary}\label{cor:smoothness-of-nondegenerate-locus}
    The nondegenerate locus $V_\G^\circ$ is $\T\times_S \G\times_S \G$-stable and contains $\G_+$ by construction. Further, the abelianization $\mf a$ restricted to $V_\G^\circ$ is smooth.
\end{corollary}

\begin{example}
    For $\G = \SL_{2}$, $V_\G$ can be identified with the monoid scheme of $2 \times 2$ matrices $\Mat_{2\times 2}$ and $\mf a$ can be identified with the determinant map $\det \colon \Mat_{2\times 2} \to \AA^1.$ The nondegenerate locus $V_\G^\circ$ can then be identified with $\Mat_{2\times 2} \setminus \{0\}.$
\end{example}

\subsection{Vinberg monoids in the non-split case}
Let $G$ be a (not necessarily split) reductive group over a scheme $S.$
Let $\G$ be the split form of $G$ and $\T$ the abstract Cartan of $\G$.
Let $T$ be the (possibly non-split) abstract Cartan of $G$, defined as the twist of $\T$ by the \'etale local $\Aut_{\G/S}$-torsor corresponding to $G$.
Define $G_+ \colonequals T \times_S^{Z_G} G$ and $T_\ad \colonequals T/Z_G.$
Like before, 
let $\Delta\subset \XX^\bullet(\T_\ad)$ be the set of simple roots and $\T_\ad \inj \T_\ad^+ \colonequals \Spec \OO_S[\XX^\bullet(\T_\ad)_{\text{pos}}]$ the obvious toric embedding, where $\XX^\bullet(\T_\ad)_{\text{pos}}$ is the submonoid generated by $\Delta.$
Finally, define $T_\ad \inj T_\ad^+$ as the twist of $\T_\ad \inj \T^+_\ad$ by the \'etale local $\Aut_{\G/S}$-torsor corresponding to $G$. Here, $\T^+_\ad$ is a possibly nontrivial affine space bundle over $S.$

\begin{theorem}
    There is a reductive monoid scheme $V_G$ over $S$ which fits into a Cartesian diagram  
\[\begin{tikzcd}
    {G_+} & {V_{G}} \\
    {T_{\ad}} & T_\ad^+
	\arrow[hook, from=1-1, to=1-2]
	\arrow[from=1-1, to=2-1]
	\arrow["{\mf a}", from=1-2, to=2-2]
	\arrow[hook, from=2-1, to=2-2]
\end{tikzcd}\]
such that 
    \begin{itemize}
        \ii $G_+$ is the unit group of $V_G$,
        \ii the $T\times_S G \times_S G$ action on $G_+$ given by $T$
        acting on the first component and $G$ acting on itself by left and
        right multiplication extends to an action
        on $V_G$,
        \ii $\mf a$ is faithfully flat and equivariant.
    \end{itemize}
\end{theorem}

\begin{proof}
This is Theorem \ref{thm:vinberg-main} applied to the split form $\G$ followed by \'etale descent.
\end{proof}

\begin{theorem}
    There is a dense open subscheme $V^\circ_G\subset V_G$ containing $G_+$, called the {\em nondegenerate locus}, such that 
    \begin{itemize} 
        \ii $V^\circ_G$ is $T\times_S G\times_S G$-stable, and
        \ii $\mf a$ restricted to $V^\circ_G$ is smooth.
    \end{itemize}
\end{theorem}

\begin{proof}
Via \'etale descent, this reduces to Corollary \ref{cor:smoothness-of-nondegenerate-locus} for $\G$.
\end{proof}

\section{Proof of Theorem
\ref{thm:main}}\label{sec:proof-of-theorem-main}
\subsection{Preliminary reductions}
Let $\G$ be the split form of $G$. Fix a pinning of $\G.$ 
This canonically determines a presentation of the automorphism group scheme $\Aut_{\G/S}$ as a semidirect
product of $\G_\ad$ and $\Out_{\G/S}$ and hence a unique quasi-split inner form $\Gqs$ of $G$
endowed with a quasi-pinning\footnote{A quasi-pinning is the data of a Killing pair $(\B,\T)$ and a section
$s \in \HH^0(\Dyn \Gqs, \mf g^{\mathfrak D})^\times$ where $\Dyn \Gqs$ is the scheme
of Dynkin diagrams of $\Gqs$ and $\mf g^{\mathfrak D}$ is a certain line bundle on
$\Dyn \Gqs.$
When $\Gqs$ is split and $\Delta$ a system of simple roots, $\Dyn \Gqs\simeq S \times \Delta$ and $\mf
g^{\mf D}$ is the line bundle which restricts to the eigenspace $\mf g_\alpha$ over $S\times
\{\alpha\}$ for each $\alpha \in \Delta$.
That is, the notion of quasi-pinning  coincides with that of pinning for
split reductive groups. See \cite[Expos\'e XXIV, \S3.9]{sga3} for more details.}
\cite[Expos\'e XXIV, Corollaire 3.12]{sga3}.

Since $G$ is assumed to be isotrivial, the \'etale local
$\G_\ad$-torsor corresponding to $\Gqs$ is isotrivial too. As
finite \'etale morphisms satisfy effective descent for projective
schemes, it suffices to prove the theorem for $\Gqs.$ Indeed, this
essentially boils down to the fact that finite group quotients of
projective schemes are representable by projective
schemes.

Choose a finite \'etale Galois cover $S' \to S$ splitting $\Gqs$ with Galois
group $\Gamma.$
The data of $\Gqs$ is then equivalent to the data of a $\Gamma$-action on the $\Gqs\times_S S'$. 
Thus, it suffices to find an equivariant compactification equipped an action of
$\Gamma$ extending that on $\Gqs\times_S S'$.
The quasi-pinning on $\Gqs$ induces a pinning on $\Gqs \times_S
S'.$ Note that $\Gamma$ additionally acts on the pinning of $\Gqs\times_S S'$.

From now, we replace our setup with a pinned reductive group $(\G,\B,\T,
\{u_\alpha\}_{\alpha\in\Delta})$ over $S$
equipped with an action of a finite group $\Gamma.$

\subsection{Cox-Vinberg hybrid}
We perform the Cox-Vinberg construction introduced in \cite[\S6]{thaddeus}
in a $\Gamma$-equivariant fashion.
As usual, let $\XX_\bullet(\T)$ be the coweight lattice,
$\XX_\bullet(\T)^+_{\QQ}$
the dominant chamber, and $W$ the Weyl group.
Of course, the dominant chamber is $\Gamma$-stable.

\begin{lemma}\label{lem:good-fan}
    There exists a $\Gamma$-stable fan $\Sigma$ which is a subdivision of the
    rational polyhedral set $\XX_\bullet(\T)^+_\QQ$
    such that $W\Sigma$, the $W$-saturation of $\Sigma$, is smooth and
    projective.
\end{lemma}
\begin{proof} 
    Subdivide the Weyl chambers in $\XX_\bullet(\T)$ to obtain a projective fan
    $\Sigma'.$
Now apply \cite[Théorème 1]{tori} to $\Sigma'$ with $W\times \Gamma$ as the finite group
acting on $\XX_\bullet(\T)$.
This yields a new smooth projective $W\times\Gamma$-stable fan $\Sigma''$
which is a subdivision of $\Sigma'.$
Then take $\Sigma = \Sigma'' \cap \XX_\bullet(\T)^+_\QQ.$
\end{proof}

\begin{remark}\label{rem:fan-for-semisimple}
    When $\G$ is semisimple, we have the canonical choice of taking $\Sigma$ to be the fan consisting of the single cone $\XX_\bullet(\T)^+_\QQ.$
    Note that this fan does not depend on the
    choice of Galois cover $S' \to S.$
\end{remark} 

Choose primitive lattice generators $\beta=\{\beta_i\}_{i\in I}$ of all the rays in $\Sigma.$
The finite group $\Gamma$ stabilizes $\beta$ and hence lifts to an action on the finite set $I.$
The $\beta_i$'s induce
monoid homomorphisms $\ol \beta_i \colon\AA^1_S \to \T_\ad^+.$
Multiplying these, we get a monoid homomorphism $\AA^I_S \to \T^+_\ad.$
Define $V_{\G,\, \beta}$ so that the following square is Cartesian:
\[\begin{tikzcd}
	{V_{\G,\beta}} & {V_\G} \\
	{\AA_S^I} & {\T_\ad^+}
	\arrow[from=1-1, to=1-2]
	\arrow[from=1-1, to=2-1]
	\arrow["{\mf a}", from=1-2, to=2-2]
	\arrow[from=2-1, to=2-2]
\end{tikzcd}\]
Then $V_{\G,\beta}$ is a $\Gamma$-equivariant reductive monoid scheme over $S$ such that all diagrams in the above square are
$\Gamma$-equivariant. Also, it has 
$\GG_{\mathrm{m},S}^I\times_S \G$ as its group of units.
For any $\sigma \subseteq I$, let $U_\sigma \colonequals \{x \in \AA^I_S \colon
x_i \ne 0 \text{ if }i\notin \sigma\}$. 
Then let $\AA_{S,\beta}^\circ$ be the union of all $U_\alpha$ such
that $\langle\beta_i \colon i\in\sigma\rangle$ is a cone in $\Sigma.$ 
Define $V_{\G,\beta}^\circ \colonequals \AA^\circ_{S,\beta} \times_{\T^+_\ad}
V_{\G}^\circ.$ This is called the {\em
nondegenerate locus} in \cite{thaddeus}.

\subsection{Compactification as a GIT quotient}
We apply geometric invariant theory developed over general bases in
\cite{seshadri}.
As in \cite[\S8]{thaddeus}, we consider quotients 
$V_{\G,\beta}\sslash_\rho \GG_{\mathrm{m},S}^I$ with respect to a
suitable linearization $\rho$ on the trivial line bundle.
Note that the formation of such GIT quotients is compatible with arbitrary base change and in
particular, passing to geometric fibers,
by virtue of {\em linear} reductivity of tori.
Semistable and stable loci are defined as open subschemes and their
formation commutes with arbitrary base change essentially by construction
\cite[\S II]{seshadri}.
In particular, one can use Hilbert-Mumford criterion along geometric
fibers to identify stable and semistable geometric points.
Therefore, the same argument
as in the proof of  \cite[Theorem 8.1]{thaddeus} works to show that
there is a linearization $\rho$ such that $V_{\G,\beta}^{\circ}$ is the
semistable (and stable) subscheme.
As a result, the GIT quotient $\ol \G \colonequals V_{\G,\beta}\sslash_\rho \GG_{\mathrm{m},
S}^I$ contains $\G$ as a fiberwise-dense open subscheme by virtue of
compatibility with formation of this quotient and restricting to geometric points.
It also acquires an action of $\Gamma$, being a geometric quotient of
$V_{\G,\beta}^\circ$ by $\GG^I_{\mathrm m, S}$.

\begin{proposition}\label{prop:separated-and-finitely-presented}
    $\ol\G$ is separated and finitely
    presented over $S.$
\end{proposition}

\begin{proof}
    \newcommand{\Gsplit}{\G_{\mathrm{Ch}}}
    Let $\Gsplit$ be the Chevalley group scheme
    over $\Spec \ZZ$ for $\G.$
    The same fan $\Sigma$ as in Lemma
    \ref{lem:good-fan} can be used to produce a
    compactification $\Gsplit\inj\ol\Gsplit$ so that $\ol \G = \ol\Gsplit \times_\ZZ S.$
    Therefore, it suffices to check that
    $\ol\Gsplit\to\Spec \ZZ$ is separated and finitely
    presented, which is readily true by standard
    properties of GIT as $\ZZ$ is universally
    Japanese (c.f.
    \cite[\S4]{seshadri}).
\end{proof}

\begin{proposition}\label{prop:smoothness}
    $\ol\G$ is $S$-smooth.
\end{proposition}

\begin{proof}
Proposition \ref{prop:big-cell} gives an open cell
$m_\beta\colon\U_-\times_S \T\times_S \AA^I_S \times_S \U_+ \inj
V_{\G,\beta}$ by base changing $m$.
The $\G\times_S\G$ translates of $m_\beta$ cover
$V_{\G,\beta}^\circ$ since the same is true for
$m$ and $V_\G^\circ$ by definition.
This open embedding is equivariant for the obvious
action of $\GG_{\mathrm m, S}^I$ on
$\T$ and $\AA^I_S.$
Passing to GIT quotients with the same linearization
as before, we thus get an open embedding 
\[\U_- \times_S \ol \T \times_S \U_+ \inj \ol
\G,\]
where $\ol \T$ is the toric scheme for $\T$
corresponding to the fan
$\Sigma$.
This toric scheme can be constructed as $\T\times_S
\AA^I_S \sslash_\rho \GG_{\mathrm m, S}^I$ with
semistable (and stable) subscheme equal to
$\T\times_S \AA_{S,\beta}^\circ$
(see, for e.g., \cite[\S5.1]{cox}).
Since $W\Sigma$ is smooth by construction
(Lemma \ref{lem:good-fan}), so is $\Sigma.$ 
Therefore, $\ol\T$ is smooth.
The desired result follows by
Lemma \ref{lem:flatness-of-saturation} since the
$\G\times_S\G$-saturation of
$\U_-\times_S\ol\T\times_S\U_+$ is $\ol\G.$ 
\end{proof}

The proof of Proposition \ref{prop:smoothness} also shows

\begin{proposition}\label{prop:open-cell}
There is an open cell 
\[\Omega\colon\U_- \times_S \ol \T \times_S \U_+ \inj \ol \G\]
whose $\G \times_S \G$-saturation is $\ol \G.$
\end{proposition}

What remains is to check projectivity of $\ol \G$. 
In \cite{thaddeus}, properness is shown by realizing $\ol\G$ as
the coarse space associated to a certain proper moduli stack of
bundles. 
We resort to an alternative approach as we do not have a modular interpretation at
hand.

\begin{proposition}
    $\ol\G$ is $S$-projective.
\end{proposition}

\begin{proof}
    As in the proof of Proposition
    \ref{prop:separated-and-finitely-presented}, we
    may assume $S = \Spec \ZZ.$
    Because $\ol\G$ is constructed as a GIT
    quotient, it suffices to check properness.
    We use valuative criterion for properness.
    For this, we may assume that $S$ is the
    spectrum of a discrete valuation ring
    $R$ with fraction field $K$. 
    By \cite[Lemma 4.1.1]{modelsofcurves}, it is enough to check that any $K$-point $x_K \in
    \G(K)$ extends to an $R$-point of $\ol\G$.
    Using the Cartan decomposition of $\G(K)$, we
    can assume $x_K=\lambda(\pi)\in \T(K)$ for some
    dominant coweight $\lambda$ and 
    uniformizer $\pi.$ 
    Such a point
    induces an $R$-point of $\T^+_\ad$ which in
    turn induces an $R$-point $x_R$ of $V_\G^\circ$ by
    applying the canonical section $\ol{\mf s}$
    (c.f. \S\ref{sec:nondegenerate-locus}).
    Base changing, we get a map
    $\AA^I_S\times_{\T^+_\ad}x_R \to V_{\G,\beta}$
    and choosing an arbitrary $S$-morphism of
    affine spaces
    $\AA^I_S\to
    \T^+_\ad$, not necessarily a section,
    we get a map $x_R \to V_{\G,\beta}.$
    By construction, the generic point of $x_R$
    lies inside the  unit group $\GG_{\text
    m,S}^I\times_S \G$ where its second component is $x_K.$
    Then the image of $x_R$ along the GIT quotient
    morphism gives the desired point.
\end{proof}

\begin{remark}
It is likely possible to obtain a complete
combinatorial classification of all equivariant compactifications of an
isotrivial reductive group scheme by following the
methods of
\cite{thaddeus}, but we do not pursue this here.
\end{remark}

\section{The semisimple case}
We prove Theorem \ref{thm:semisimple} in this section.
Let $G$ be a semisimple reductive group scheme
over $S.$
By \cite[Expos\'e XXIV, Théorème 4.1.5]{sga3}, $G$ is locally
isotrivial.
This allows us to carry out the construction of
\S\ref{sec:proof-of-theorem-main} Zariski
locally on $S$ where we always choose $\Sigma$
according to Remark \ref{rem:fan-for-semisimple}.
Since this choice is independent of the finite
\'etale Galois covers, they glue to give $\ol G \to S.$
By \cite[Theorem 8]{vinberg} and \cite[Proposition 1.3]{bouthier} and compatibility of
such GIT quotients with base change, geometric fibers of
$\ol G_\ad$ are indeed classical wonderful
compactifications of de Concini and Procesi (also see \cite[Theorem 5.4]{thaddeus}).
We have shown that $\ol G_\ad$ is smooth projective
over $S$.
There is a morphism $\ol G \to
\ol G_\ad$ coming from a natural map of Cox-Vinberg
monoids, extending the central isogeny $G \to
G_\ad.$ 
By 
\cite[Lemma 9.2]{thaddeus}, $(\ol G)_s$ is the
normalization of $G_s$ in $(\ol G_\ad)_s$ for each
geometric point $s \to S.$ 
What remains is to prove the statement about the
boundary.

\begin{assumption}
From now on, we assume $G$ is adjoint.
\end{assumption}

\begin{proposition}
    $\ol G\setminus G$ is the union of $S$-smooth relative
    effective Cartier divisor with normal
    crossings.
\end{proposition}

\begin{proof} 
    By construction, the formation of $\ol
G\setminus G \inj \ol G$ commutes with base
change on $S.$
Since smoothness, being a relative divisor, and having normal crossings
are \'etale local on the base, we may assume that
$G$ is split and $S = \Spec \ZZ$.
By Proposition \ref{prop:open-cell}, there is an open cell 
\[\Omega\colon\mU_- \times_S \ol T \times_S \mU_+ \inj \ol G\]
where the toric
embedding $T \inj \ol T$ looks like $\prod_{\alpha
\in \Delta}\GG_{\mathrm m, S} \inj \prod_{\alpha\in\Delta}\AA^1_S.$
The complement $\ol T\setminus T$ is clearly
an $S$-flat divisor, therefore by Lemma
\ref{lem:flatness-of-saturation}, so is the whole
boundary $\ol G\setminus G$.
By \cite[\href{https://stacks.math.columbia.edu/tag/062Y}{Tag
062Y}]{stacks} and the fact that translates of
$\Omega$ cover $\ol G$ when $S$ is an
algebraically closed point, it follows that
$D\colonequals \ol G\setminus G$
is a relative effective Cartier divisor. 
Since $G$ is connected, irreducible components, say
$D_i$, of
this divisor are $G\times_S G$-stable.
Since $\Omega$ is dense in $\ol
G$, it intersects each $D_i$ in a dense open
subscheme, which implies that $D_i$ is the $G\times_S G$-saturation of
the dense open $\Omega\cap D_i$ inside $D_i.$
Hence, we obtain that $D_i$ is $S$-smooth by Lemma
\ref{lem:flatness-of-saturation}.
What remains is to check that $D = \sum_i D_i$ has
relative normal crossings.
We know that $D_s$ has normal crossings for each
geometric point $s \to S$ since translates of
$\Omega_s$ cover $\ol G_s.$
The desired result now follows from Lemma
\ref{lem:normal-crossings}.
\end{proof}

\begin{remark}
    In fact, we end up proving a bit more when $G$
    is split: irreducible components of the
    boundary divisor are indexed by a set of system
    of simple roots $\Delta$ and $G\times_S
    G$-stable subschemes of $\ol G \setminus G$ correspond to subsets of $\Delta.$ 
\end{remark}

\begin{lemma}\label{lem:normal-crossings}
    Let $X$ be a smooth $S$-scheme equipped with a
    relative effective Cartier divisor $D$
    with $S$-smooth irreducible components.
    Assume that for each geometric point $s\in S$,
    $D_s$ has normal crossings in $X_s$. Then $D$ has relative normal crossings over $S.$
\end{lemma}

\begin{proof}
    Let $x \in X$ be a point with image $s \in S$, and $D_1,\ldots,D_n$
    the irreducible components of $D$ passing
    through $x.$
    Choose an affine open neighborhood $U$
    of $x$ such that $D_i
    \cap U$ is cut out by $f_i = 0$ for
    $f_i\in \HH^0(U, \OO_X).$
    The images of $f_i$ in $\Omega_{X/S} \tensor
    \kappa(x)$ are linearly independent because
    $\Omega_{X/S}\tensor\kappa(x) \simeq
    \Omega_{X_s/s}\tensor \kappa(x)$ 
    and the same is true in $\Omega_{X_{\ol s}/\ol s}\tensor
    \ol{\kappa(x)}$ by assumption, where $\ol s$ is an
    arbitrary geometric point lying over $s.$
    Let $N$ be the rank of $\Omega_{X/S}$ at $x.$
    By Nakamaya and possibly shrinking $U$, extend these to a set of
    functions
    $\{f_1,\ldots,f_n,f_{n+1},\ldots,f_N\}$ each
    of which vanish at $x$ such that their
    differentials
    form a basis for the trivial vector bundle
    $\Omega_{U/S}.$
    We thus obtain an {\em \'etale} $S$-morphism 
    \[f\colon U \xrightarrow{(f_1,\ldots,f_N)}
    \AA^N_S\]
    such that $D\cap U$ is the preimage of the
    relative normal crossing divisor
    $x_1x_2\cdots x_n = 0$ in $\AA^N_S.$
\end{proof}

\begin{proposition}\label{prop:agree-with-shangli}
    $\ol G$ agrees with
    the equivariant compactification $\mathcal X$
    of \cite[Theorem 1]{shangli}.
\end{proposition}

\begin{proof}
    Firstly, it suffices to check this for quasi-split
    $G$ because our $\ol G$ and the
    compactification $\mathcal X$
    of \cite[Theorem 1]{shangli} is obtained by
    first making the same constructions for split
    $G$ and then performing the obvious inner
    twist.
    The construction for quasi-split $G$ is
    basically the same as the construction for split
    $G$ where everything is equipped with the action
    of a finite abstract group preserving a
    pinning. Thus, we reduce to the split case.
    By Proposition \ref{prop:open-cell}, we have an open
    cell $\Omega \colon \mU_- \times_S \ol T
    \times_S \mU_+ \inj \ol G$ whose $G\times_S
    G$-saturation is $\ol G.$
    The method of \cite{shangli} starts by defining
    a {\em rational} $G\times_S G$-action $\pi$
     on
    $\Omega$ \cite[Theorem 3.4]{shangli} and then
    defining $\mathcal X$ as an appropriate fppf
    sheaf quotient of $G\times_S \Omega \times_S
    G.$
    Due to the uniqueness assertion of
    \cite[Theorem 3.4]{shangli}, we reduce to
    checking that the rational $S$-morphism 
    $G\times_S \Omega \times_S G\dasharrow \Omega$
    induced by the action map
    $G\times_S \Omega \times_S
    G\to\ol G$ given by $(g_1,\omega,g_2)\mapsto
    g_1\omega g_2$ satisfies the conditions of
    {\em loc.\!\! cit.}, but this is clear.
\end{proof}

\section{A torus admitting no equivariant
compactification}\label{sec:bad-torus}
We prove Theorem \ref{thm:counterexample} in this section.
Let $k$ be an algebraically closed field.
We recall the construction of \cite[Expos\'e X, \S1.6]{sga3}.
Let $S_1$ be the N\'eron $1$-gon, obtained by glueing
sections $0$ and $\infty$ of $\PP^1_k.$ 
It can be realized as the nodal cubic curve $S_1$ equipped with the
normalization
$\pi_1\colon\PP^1_k\to S_1.$
Let $S_\infty$ be the N\'eron $\infty$-gon which
comes equipped with a finite morphism
$\pi_\infty\colon
\PP^1_{k}\times \ZZ\to
S_\infty$ which glues the $\infty$-section of
$\PP^1_k \times\{i\}$ with the $0$-section of
$\PP^1_k\times \{i+1\}.$
There is an infinite connected \'etale Galois cover $S_\infty
\to S_1$ with Galois group $\ZZ$ which is covered by
the trivial $\ZZ$-torsor
$\PP^1_{k}\times \ZZ \to \PP^1_k$
where $\ZZ$ acts on the source via
$j\cdot(x,i)=(x,i+j)$.

Define an action of $\ZZ$ on the constant torus
$\GG_{\mathrm m, k}^2 \times S_\infty$ by
$1\cdot (t, x)=(Mt, 1\cdot x)$ where $M$ is the
{\em infinite} order automorphism of $\GG^2_{\mathrm m,
k}$ given by $(t_1,t_2)\mapsto (t_1,t_1t_2).$
By Galois descent, we thus obtain a rank $2$ torus
$T_{S_1}
\to S_1$.
Since $T_{S_1}$ has infinite monodromy by
construction, $T_{S_1}$ is a
quasi-isotrivial\footnote{This means it is split by
an \'etale cover.} torus that
is not isotrivial.
Alternatively, $T_{S_1}$ can be constructed by taking the
constant torus $\GG_{\mathrm m, k}^2 \times
\PP^1_k\to \PP^1_k$ and glueing the fibers over
$0$ and $\infty$ via the automorphism $M.$

\begin{proposition}
    There is no projective $S_1$-scheme $\ol T_{S_1}$
    containing $T_{S_1}$ as a fiberwise dense open
    subscheme such that the translation action of
    $T_{S_1}$ on itself extends to $\ol T_{S_1}.$
\end{proposition}
Assume the contrary that such a $\ol T_{S_1}$ exists.
First, replace $\ol T_{S_1}$ by its reduction so that
    it's a genuine $k$-variety-- we know that it is
    irreducible because it contains a torus as an
    open dense subvariety.
    Let $X$ be the base change of $\ol T_{S_1}$ along the normalization $\pi_1 \colon \PP^1_k \to S_1.$
    Let $T \colonequals \torus \times_k \PP^1_k$ be the open dense torus inside $X.$

    Since $X$ is irreducible, it must be flat over $\PP^1_k.$
    Set $T \colonequals \GG_{\mathrm m, \PP^1_k}^2$.
    We use the following result of Brion:

    \begin{theorem}[\protect{\cite[Theorem 4.8]{brion}}]\label{thm:brion-etale-cover}
    Let $G$ be a split torus over a field $k$. Then every quasiprojective $k$-variety $X$ equipped with an action of $G$ admits a finite \'{e}tale $G$-equivariant cover $f \colon Y \to X$, where $Y$ is the union of open affine $G$-stable subvarieties. 
\end{theorem}

We may view $X$ as a projective variety with $\GG_{\mathrm m, k}^2$-action by identifying $T\times_{\PP^1_k} X \simeq \GG_{\mathrm m, k}^2 \times_k X$ in the source of the action map.
By applying Theorem \ref{thm:brion-etale-cover}, we get a \torus-equivariant finite \'etale cover $\pi\colon Y \to X$
such that $Y$ is the union of open affine \torus-stable subvarieties.

\begin{lemma}
    $\pi \colon \pi^{-1}(T) \to T$ is a trivial cover.
\end{lemma}

\begin{proof}
    Any finite \'etale cover of $T$ comes from a finite \'etale cover $U \to \torus$. Indeed, the projection $T\to \torus$ is a $\PP^1_k$-bundle and hence induces an isomorphism on \'etale fundamental groups by the homotopy exact sequence. 
    A connected finite \'etale cover $U \to \torus$ must be an \'etale self-isogeny induced by a linear endomorphism of the character lattice. Equivariance forces such an isogeny to be an isomorphism.
\end{proof}

Label the components of $\pi^{-1}(T)$ as $T_1, T_2,\ldots, T_n$. Each of these are abstractly isomorphic to $T$ and map down to $T\subset X$ via identity.
The \torus-action on $\ol T_i$ upgrades to a $T_i$-action extending the one on $T_i$ by itself.
The union $\bigcup_{i} T_i\subset Y$ is a $\PP^1_k$-fiberwise open-dense subscheme since the same is true for $T\subset X.$
Denote by $t$ a closed point of $\PP^1_k$.
The fiber $Y_t$ has irreducible components given by the  (scheme-theoretic) closures of $(T_1)_t, (T_2)_t,\ldots, (T_n)_t.$
In particular, each fiber of $Y$ has $n$ irreducible components.
Every closure $\ol T_i,\, 1\le i \le n,$ is faithfully flat over $\PP^1_k$ and hence has fibers of pure dimension $2.$
Therefore, each $(\ol T_i)_t$ is the union of closures of a nonempty subset of $\{(T_1)_t, (T_2)_t,\ldots, (T_n)_t\}.$
For $i\ne j$, $(\ol T_i)_t$ and $(\ol T_j)_t$ cannot have an irreducible component in common, say the closure of $(T_h)_t$, because $Y_t$ is generically reduced. 
Alternatively, if $\xi$ is the generic point of $(T_h)_t$,
the local ring $\OO_{Y,\xi}\simeq \OO_{T_h,\xi}$ is a DVR and hence cannot contain more than one minimal prime.
We thus obtain:

\begin{proposition}
    Every $\ol T_i$ is a flat projective variety over $\PP^1_k$ containing the split torus $T_i$ as a fiberwise dense open subvariety such that the action of $T_i$ on itself extends to $\ol T_i$ in a way that $\ol T_i$ can be covered by \torus-stable open affine subvarieties. Furthermore, there is an isomorphism between the normalizations of the $0$-fiber and $\infty$-fiber of $\ol T_i$ restricting to $(x,y)\mapsto (x,xy)$ on \torus.
\end{proposition}

\begin{proof}
    The first part is clear from the previous discussion.
    For the second part, note that there is a finite birational morphism $(\ol T_i)_t \to X_t$ because it restricts to identity on the toral part and hence induces an isomorphism on normalizations by Zariski's main theorem.
\end{proof}

Set $W \colonequals \ol T_1$ and rename $T_1$ as $T$ for ease of notation.
 Denote $W \to \PP^1_k$ by $f.$
Let $U\subset W$ be a nonempty \torus-stable open affine subvariety.
Then $V \colonequals f(U)$ is a nonempty open subvariety of $\PP^1_k.$
Of course, $U$ is $\PP^1_k$-fiberwise open and hence intersects $T_V = \GG_{m, V}^2$ fiberwise.
Due to \torus-invariance, $U$ must contain whole of $T_V.$
The \torus-action on $U$ can be upgraded to an action of $T_V.$
Therefore, 

\begin{proposition}\label{prop:local-toric-scheme}
    Let $\ul x \colon \Spec \OO_{\PP^1_k, x} \to \PP^1_k$ be the local scheme at a closed point $x\in \PP^1_k.$
Then $W_{\ul x}$
    can be covered by $T_{\ul x}$-stable affine open neighborhoods of $T_{\ul x}\subset W_{\ul x}.$
\end{proposition}

\begin{lemma}\label{lem:normalization}
    Let $X$ be an integral affine scheme over a discrete valuation ring $R$ which contains a split $R$-torus $T$ as a \textup{fiberwise dense} open subscheme such that the action of $T$ on itself extends to $X$.
    Then the fibers of the normalization $\widetilde X$  as an $R$-scheme are irreducible and normal.
\end{lemma}

\begin{proof}
    Let $M$ be the character lattice of $T$. 
    Such an affine toric scheme $X$ is given by a graded sub-$R$-algebra $A$ of the $M$-graded group ring $R[M]$, which in turn is equivalent to the data of a finitely generated submonoid $Q$ of $M$ which generates $M$ as a group.
    The last bit ensures that $X$ contains $\GG_{\mathrm m, R}^n$ as a open dense subscheme.
    The normalization of $A$ then corresponds to the saturation monoid $\widetilde Q$ of $Q$ defined as $\{a \in M \colon na \in Q\text{ for some }n \in \ZZ\}.$
    The special fiber of $\widetilde X$ is then the spectrum of the monoid ring $k[\widetilde Q]$, where $k$ is the residue field of $R.$
    Since $k[\widetilde Q]$ is a subring of the integral domain $k[M]$, it follows that the special fiber of $\widetilde X$ is irreducible. Since $\widetilde Q$ is saturated, $k[\widetilde Q]$ is also integrally closed.
\end{proof}

Let $\widetilde W_{\ul x} \to W_{\ul x}$ be the normalization.
By Proposition \ref{prop:local-toric-scheme} and Lemma \ref{lem:normalization}, it follows that $\widetilde W_\ell$ is a flat projective normal $T_{\ul x}$-toric scheme over $R\colonequals k[t]_{(t)}$ with irreducible normal fibers.
Consider $\widetilde W_{\ul 0}$ and $\widetilde W_{\ul \infty}$.
The generic fibers of these are identified and there is an isomorphism between their special fibers which restricts to $(x,y)\mapsto (x,xy)$ on respective toral parts.
    By the theory of normal toric schemes over DVRs
    \cite[item e) at p. 192]{mumford},
    these are classified by two {\em complete} rational polyhedral fans $\Sigma_1$ and 
    $\Sigma_2$ in $\RR^2 \times \RR_{\ge 0}$, respectively.
    Let $\pi \colon \RR^2\times \RR_{\ge 0}\to \RR^2$ be the natural projection where we often view the target as sitting inside $\RR^2\times \RR_{\ge 0}$ as $\RR^2\times \{0\}.$
    We recall the following two facts: 
    \begin{itemize}
        \ii By \cite[item e') at p. 192]{mumford},
        the recession fan of $\Sigma_i$, defined as the image of $\Sigma_i \cap (\RR^2 \times \{0\})$ along $\pi$, classify the respective generic fibers, and are therefore equal, to say $\Sigma_\del$.
        \ii There is an embedding of fans $\{0\}\times \RR_{\ge 0}\inj \Sigma_i$ corresponding to the open embedding of toric schemes $\mathbb G_{\mathrm m,R}^2\inj \widetilde W_{\ul 0}$ and  $\mathbb G_{\mathrm m,R}^2\inj \widetilde W_{\ul \infty}$.
        Indeed, the toric scheme $\GG_{\mathrm m, R}^2 \to \Spec R$ is classified by the cone $\{0\}\times \RR_{\ge 0}.$
        Therefore, the components of the special fiber containing $\GG_{\mathrm m, k}^2$ are classified by the (complete) fan $\Delta_i\colonequals \{\pi(\sigma)\colon \sigma \in \Sigma_i,\, \{0\}\times \RR_{\ge 0} \subseteq \sigma\}$ as a toric variety. Furthermore, there is a one-to-one correspondence between irreducible components of the special fiber  and vertices of the polyhedral complex $\Sigma_i \cap (\RR^2 \times \{1\})$ (c.f. \cite[Proposition 7.15]{gubler} or \cite[Remark 3.5.9]{arithmetic-geometry-of-toric-varieties}).
    \end{itemize}
    We thus conclude that every ray in $\Sigma_i$ not equal to $\{0\} \times \RR_{\ge 0}$ must be contained in the boundary $\RR^2\times \{0\}$. That is, $\widetilde W_{\ul x}$ is a constant family-- the base change of a normal toric variety over $k$.
    Now, the fact that there is an isomorphism between the special fibers of $\widetilde W_{\ul 0}$ and $\widetilde W_{\ul \infty}$ extending $(x,y) \mapsto (x,xy)$ on the toral part corresponds to the fact that $A(\Delta_1) = \Delta_2$ where $A = \begin{bmatrix} 1 & 1 \\ 0 & 1\end{bmatrix}$.
    However, $\Delta_1 = \Sigma_\del = \Delta_2.$
    Thus, $\Sigma_\del$ is a complete fan in $\RR^2$ which is stable under the automorphism $A.$
    There must exist a ray $\ell \in \Sigma_\del$ which is not contained in the $x$-axis for otherwise it would not be complete.
    Then $\{A^n\ell\colon n \in \ZZ\}$ is an infinite set.
    This contradicts the finiteness of $\Sigma_\del.$
    The proof is complete. 

\printbibliography
\end{document}